\documentclass[12pt]{article}
\usepackage{amsmath,amssymb,amsthm,graphicx}
\usepackage{enumitem}
\usepackage{hyperref}
\usepackage{fullpage}
\usepackage{float}
\usepackage{setspace}
\usepackage{tabu}
\usepackage{multicol}
\usepackage{longtable}

\newtheorem{theorem}{Theorem}

\newtheorem{observation}[theorem]{Observation}
\newtheorem{proposition}[theorem]{Proposition}
\newtheorem{corollary}[theorem]{Corollary}
\newtheorem{conjecture}[theorem]{Conjecture}

\begin{document}

\title{Minimum Weighted Szeged Index Trees}

\author{
Pavol Hell
\and
C\'{e}sar Hern\'{a}ndez-Cruz
\and
Seyyed Aliasghar Hosseini
}

\date{\today}

\maketitle

\begin{abstract}
Weighted Szeged index is a recently introduced extension of the well-known Szeged index. In this paper, we present a new tool to analyze and characterize minimum weighted Szeged index trees. We exhibit the best trees with up to 81 vertices and use this information, together with our results, to propose various conjectures on the structure of
minimum weighted Szeged index trees.
\end{abstract}

\section{Introduction}
Molecular descriptors \cite{ref1} are mathematical quantities that describe the structure or shape of molecules, helping to predict the activity and properties of molecules in complex experiments. In the last few years, several new molecular structure descriptors have been conceived \cite{ref2,ref3,ref4,ref5}.  Molecular descriptors play a significant role in chemistry, pharmacology, etc. Topological indices have a prominent place among molecular structure descriptors. Modeling physical and chemical properties of molecules, designing  pharmacologically active compounds and recognizing environmentally hazardous materials are just a few applications of topological indices, see~\cite{ref6}.  One of the most important topological indices is the \emph{Szeged index}. The Szeged index of a connected graph $G$, $Sz(G)$, is defined as
$$Sz(G) = \sum_{e = uv\in E(G)} n_u^G(e)n_v^G(e),$$
where $n_u^G(e)$ is the number of vertices in $G$ that are closer to $u$ than $v$.
Similarly, the \emph{weighted Szeged index} of a connected graph $G$, $wSz(G)$, introduced
by Ili\'{c} and Milosavljevi\'{c} in \cite{ilic}, is defined as
$$wSz(G) = \sum_{e = uv\in E(G)} \left(d_G(u)+d_G(v)\right)n_u^G(e)n_v^G(e),$$
where $d_G(u)$ is the degree of $u$ in $G$.

Recently some researches have studied the weighted
Szeged index of special classes of graphs and graph
operations \cite{jan,ref10,ref11,ref12}. One of the
main related open problems is characterizing a graph
of fixed order with minimum weighted Szeged index.
Regarding this problem, if true, the following
conjecture would substantially restrict the structure
of such graphs.

\begin{conjecture}\cite{jan} \label{min-tree}
For $n$-vertex graphs the minimum weighted Szeged
index is attained by a tree.
\end{conjecture}
Unlike many other topological indices, for example
ABC-index \cite{ABCJCTB,krag_tree}, it is not easy
to prove if the graph with minimum weighted Szeged
index on a fixed number of vertices is a tree.
In \cite{jan} some properties of minimum weighted
Szeged index trees are described and the list of
minimum trees on at most $25$ vertices is presented.

We refer the reader to \cite{bondy2008} for basic
graph theory terminology.   Throughout this work,
we will use $n$ to denote the number of vertices
of a graph.   The rest of the work is organized
as follows.   In Section \ref{sec:ending} we
introduce the concept of ending branch, present
some basic results about their structure and
use these results to determine the structure of
ending branches for some orders.   We present
computational results on the structure of optimal
trees for the weighted Szeged index, as well as
the structure of optimal ending branches in
Section \ref{sec:computational}.   We cover
trees up to 81 vertices; it is not hard to
computationally push this bound, but we only
intend to present the readers with a dataset
that might give a fair idea of what is happening
with these structures, and not giving the most
complete list possible.  Finally, we present
conclusions and some conjectures based on our
observations in Section \ref{sec:conc}.

\section{Ending Branches}
\label{sec:ending}
Let $T$ be a tree with the smallest weighted Szeged index on a fixed number of vertices and let $R$ be a vertex of highest degree in $T$. Note that removing any edge from $T$ results in two connected components. We will call the the one that does not contain $R$ an \emph{ending branch}. In this section we will study the ending branches that result in small weighted Szeged index trees. By the weighted Szeged index of an ending branch we mean the sum of the weighted Szeged index of all edges in the ending branch along with the half-edge connecting the ending branch to the rest of the tree. Note that to calculate the weighted Szeged index of an ending branch, it is sufficient to have the total number of vertices of the graph besides the structure of the ending branch. The following proposition clarifies the idea.

\begin{proposition}
Let $n$ be a positive integer and let $T$ be a tree with the
smallest weighted Szeged index on $n$ vertices.  If $uv$ is
an edge of $T$ and $v$ is the root of the associated
ending branch, then the weighted Szeged index of the ending
branch does not depend on the degree $d_u$ of $u$.
\end{proposition}

\begin{proof}
Consider two edges $uv$ and $uv'$ of $T$ and
suppose that the associated ending branches with roots $v$
and $v'$ have the same number of vertices.
The difference between the weighted Szeged index of these two graphs will depend on $d_u$ only on the edge $uv$ which is
$$(d_u+d_v)n_u^v n_v^u - (d_u+d_{v'})n_u^{v'} n_{v'}^u,$$
and if the number of vertices in the ending branch is the same then it simplifies to
$$(d_v-d_{v'})n_u^v n_v^u,$$
which does not depend on $d_u$, as claimed.
\end{proof}

Now we can compare the weighted Szeged index of different ending branches. By a {\em minimal ending branch on $k$ vertices} we mean the ending branch that has the smallest weighted Szeged index among all possible ending branches of $k$ vertices.

\begin{corollary}\label{cor:end}
In any minimum weighted Szeged index tree all ending branches are also minimal.
\end{corollary}

This allows us to focus on the best ending branches of a given size in a tree.

\begin{observation}
There is only one possible ending branch of size 1 and one of size 2.
\end{observation}
\begin{figure}[htb]
    \centering
    \includegraphics{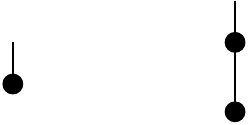}
    \caption{Ending branches of size 1 and 2.}
    \label{fig:1-2}
\end{figure}
Let $T_i$ be a best ending branch of size $i$. For $T_1$ and $T_2$ we have:
$$wSz(T_1)=n-1,$$
$$wSz(T_2)=(n-1) + 2\times 2 \times (n-2) = 5n-9.$$

Figure \ref{fig:3} shows the two  possible ending branches of size 3.
\begin{figure}[htb]
    \centering
    \includegraphics{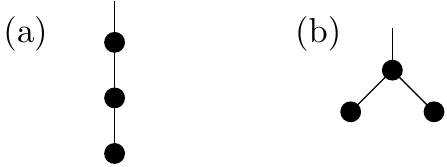}
    \caption{Possible ending branches of size 3.}
    \label{fig:3}
\end{figure}
\begin{proposition}
The 3-ray is a better ending branch than a vertex of degree 3 with 2 leaves attached to it, i.e. in Figure \ref{fig:3}, (a) is smaller than (b).
\end{proposition}

\begin{proof}
Let $n$ be the number of vertices of $T$, then we have:
$$wSz(T_a)=(1+2)(n-1)(1)+(2+2)(n-2)(2) + 2(n-3)(3)=17n-37,$$
$$wSz(T_b)=2(1+3)(n-1)(1)+3(n-3)(3)=17n-35.$$
Therefore for any $n$ we  have that $wSz(T_a)<wSz(T_b)$.
\end{proof}

Note that $wSz(T_a)$ and $wSz(T_b)$ can be rephrased as
$$wSz(T_a)=2\times3\times(n-3)+ 2\times 2 \times (n-2) + wSz(T_2),$$
$$wSz(T_b)=3\times3\times(n-3) + 2\times 3\times (n-1) + 2wSz(T_1).$$

This idea can be generalized as follows. Let $v$ be the root of an ending branch and let $x_1, x_2, \ldots, x_k$ be its children. Also let $n_i$ be the size of the ending branch with root $x_i$ (for $i \in \{ 1, \ldots, k \}$). Then the best ending branch with root $v$ has the following weighted Szeged index:
$$wSz(T_v)=d_v\times n_v \times (n-n_v) + \sum_{i=1}^k d_v\times n_i\times(n-n_i) + \sum_{i=1}^k wSz(T_{x_i}),$$
where $n_v$ is the total number of vertices in this ending branch.
Observe that $\sum_{i=1}^k n_i =n-1$, and for every partition of $n-1$ into integers we will get a new ending branch. We need to compare all these branches to find the best ending branch on $n_v$ number of vertices.
\begin{figure}[htb]
    \centering
    \includegraphics{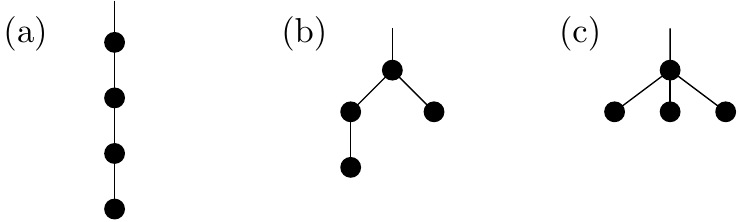}
    \caption{Possible ending branches of size 4.}
    \label{fig:4}
\end{figure}
As an example, to find the best ending branch on 4 vertices we need to find all partitions of 3 and we have $3=3 , 3=2+1, 3=1+1+1$. See Figure \ref{fig:4}. We have:
$$wSz(T_{3})=2\times 4\times (n-4) + 2\times 3\times (n-3) + 17n-35 = 31n-85,$$
$$wSz(T_{2+1}) = 3\times 4\times (n-4) + 3\times 2\times (n-2) + 3\times (n-1) + 5n-9 + n-1 = 27n - 73,$$
$$wSz(T_{1+1+1})=4\times4\times(n-4) + 3\times 4\times(n-1) + 3\times (n-1) = 31n-79.$$
Now it is easy to see that $wSz(T_{2+1})<wSz(T_3)<wSz(T_{1+1+1})$ for $n>6$. Therefore, in any minimal weighted Szeged index tree of size more than 6, if an ending branch of size 4 occurs, it is going to be $T_{2+1}$.

\section{Computational Results}
\label{sec:computational}

In this section we will use Corollary \ref{cor:end} to find minimum ending branches of higher order, say $n_v$, when the tree has $n$ vertices. Let us assume that we know all the minimum ending branches of order up to $n_v-1$.

Any ending branch has a root, say $v$, and the sum of the order of its children is $n_v-1$.
Basically, the order of children of $v$ is a partition of $n_v-1$. So we need to go over all possible partitions of $n_v-1$ and since we know the best ending branches of order up to $n_v-1$ we can just calculate and compare the best ending branch of order $n_v$.
This is when we have a fixed $n$. In general, when $n$ is a variable, for each partitioning of $n_v-1$ into integers we get a linear equation and to compare linear equations we sometimes need a bound.

Here we present the result of our computations. Table \ref{tab:branch} shows the best ending branches of size $n_v$. When $n$ is at least a specific value, then $v$ has $d_v$ children of shown order. Using the same approach we can find the minimal weighted Szeged index trees as well. Table \ref{tab:tree} shows the minimum weighted Szeged index trees of order up to 81.

Based on this calculation and as an example, the minimum weighted Szeged index tree on 67 vertices has a root of degree 4 and the root has three children of  order 16, one child of order 18. A minimum ending branch on 16 vertices (when there are more than 18 vertices in the tree) has three children of order 5. A minimal ending branch of order 5 (when there are more than 6 vertices in the tree) has two children of order 2 and there is one possible ending branch of order 2, shown in Figure \ref{fig:1-2}. A drawing of the minimum weighted Szeged index tree on 67 vertices is shown in Figure \ref{fig:5}.

\small{
\begin{longtable}{| c | c | c | l |}
 \hline
 $n_v$ & $n\geq$ & $d_v$ & children \\ [0.5ex]
 \hline\hline
2 & 2 & 1 & 1 \\
\hline
3 & 3 & 1 & 2 \\
\hline
4 & 6 & 2 & 1, 2 \\
\hline
5 & 6 & 2 & 2, 2 \\
\hline
6 & 8 & 2 & 2, 3 \\
\hline
7 & 9 & 3 & 2, 2, 2 \\
\hline
8 & 14 & 3 & 2, 2, 3 \\
\hline
9 & 11 & 2 & 3, 5 \\
\hline
10 & 14 & 3 & 3, 3, 3 \\
\hline
10* & 14 & 2 & 4, 5 \\
\hline
11 & 13 & 2 & 5, 5 \\
\hline
12 & 15 & 3 & 3, 3, 5 \\
\hline
13 & 18 & 3 & 3, 4, 5 \\
\hline
14 & 17 & 3 & 3, 5, 5 \\
\hline
15 & 18 & 3 & 4, 5, 5 \\
\hline
16 & 18 & 3 & 5, 5, 5 \\
\hline
17 & 20 & 3 & 5, 5, 6 \\
\hline
18 & 22 & 3 & 5, 5, 7 \\
\hline
19 & 25 & 3 & 5, 6, 7 \\
\hline
20 & 32 & 3 & 5, 7, 7 \\
\hline
21 & 70 & 3 & 6, 7, 7 \\
\hline
22 & 45 & 3 & 7, 7, 7 \\
\hline
23 & 51 & 3 & 7, 7, 8 \\
\hline
24 & 48 & 3 & 7, 7, 9 \\
\hline
25 & 49 & 3 & 7, 7, 10 \\
\hline
26 & 39 & 3 & 7, 7, 11 \\
\hline
27 & 48 & 3 & 7, 8, 11 \\
\hline
28 & 47 & 3 & 7, 9, 11 \\
\hline
29 & 50 & 3 & 7, 10, 11 \\
\hline
30 & 41 & 3 & 7, 11, 11 \\
\hline
31 & 45 & 3 & 8, 11, 11 \\
\hline
32 & 44 & 3 & 9, 11, 11 \\
\hline
33 & 44 & 3 & 10, 11, 11 \\
\hline
34 & 42 & 3 & 11, 11, 11 \\
\hline
35 & 54 & 3 & 11, 11, 12 \\
\hline
36 & 55 & 3 & 11, 11, 13 \\
\hline
37 & 54 & 3 & 11, 11, 14 \\
\hline
38 & 55 & 3 & 11, 11, 15 \\
\hline
39 & 45 & 3 & 11, 11, 16 \\
\hline
40 & 54 & 3 & 11, 12, 16 \\
\hline
41 & 55 & 3 & 11, 13, 16 \\
\hline
42 & 54 & 3 & 11, 14, 16 \\
\hline
43 & 55 & 3 & 11, 15, 16 \\
\hline
44 & 50 & 3 & 11, 16, 16 \\
\hline
45 & 66 & 3 & 12, 16, 16 \\
\hline
46 & 56 & 3 & 13, 16, 16 \\
\hline
47 & 55 & 3 & 14, 16, 16 \\
\hline
48 & 55 & 3 & 15, 16, 16 \\
\hline
49 & 56 & 3 & 16, 16, 16 \\
\hline
50 & 65 & 3 & 16, 16, 17 \\
\hline
51 & 64 & 3 & 16, 16, 18 \\
\hline
52 & 72 & 3 & 16, 16, 19 \\
\hline
53 & 83 & 3 & 16, 16, 20 \\
\hline
54 & 98 & 3 & 16, 16, 21 \\
\hline
55 & 121 & 3 & 16, 16, 22 \\
\hline
56 & 125 & 3 & 16, 17, 22 \\
\hline
57 & 165 & 3 & 16, 18, 22 \\
\hline
58 & 254 & 3 & 16, 19, 22 \\
\hline
59 & 506 & 3 & 16, 20, 22 \\
\hline
60 & 66 & 4 & 11, 16, 16, 16 \\
\hline
61 & 67 & 4 & 12, 16, 16, 16 \\
\hline
62 & 68 & 4 & 13, 16, 16, 16 \\
\hline
63 & 70 & 4 & 14, 16, 16, 16 \\
\hline
64 & 71 & 4 & 15, 16, 16, 16 \\
\hline
65 & 71 & 4 & 16, 16, 16, 16 \\
\hline
66 & 72 & 4 & 16, 16, 16, 17 \\
\hline
67 & 73 & 4 & 16, 16, 16, 18 \\
\hline
68 & 77 & 4 & 16, 16, 16, 19 \\
\hline
69 & 78 & 4 & 16, 16, 16, 20 \\
\hline
70 & 79 & 4 & 16, 16, 16, 21 \\
\hline
71 & 80 & 4 & 16, 16, 16, 22 \\
\hline
72 & 82 & 4 & 16, 16, 17, 22 \\
\hline
73 & 84 & 4 & 16, 16, 18, 22 \\
\hline
74 & 87 & 4 & 16, 16, 19, 22 \\
\hline
75 & 93 & 4 & 16, 16, 20, 22 \\
\hline
76 & 100 & 4 & 16, 16, 21, 22 \\
\hline
77 & 99 & 4 & 16, 16, 22, 22 \\
\hline
78 & 107 & 4 & 16, 17, 22, 22 \\
\hline
79 & 117 & 4 & 16, 18, 22, 22 \\
\hline
80 & 128 & 4 & 16, 19, 22, 22 \\
\hline
\caption{Minimal Ending branches}
\label{tab:branch}
\end{longtable}
}


 \begin{longtable}{| c | c | l |}
 \hline
 $n$ & $d_R$ & children \\ [0.5ex]
 \hline\hline
2 & 1 & 1 \\
\hline
3 & 1 & 2 \\
\hline
4 & 2 & 1, 2 \\
\hline
5 & 2 & 2, 2 \\
\hline
6 & 3 & 1, 2, 2 \\
\hline
7 & 3 & 2, 2, 2 \\
\hline
8 & 3 & 2, 2, 3 \\
\hline
9 & 4 & 2, 2, 2, 2 \\
\hline
10 & 3 & 2, 2, 5 \\
\hline
11 & 3 & 2, 3, 5 \\
\hline
12 & 4 & 2, 2, 2, 5 \\
\hline
13 & 3 & 2, 5, 5 \\
\hline
14 & 3 & 3, 5, 5 \\
\hline
15 & 4 & 3, 3, 3, 5 \\
\hline
15* & 3 & 4, 5, 5 \\
\hline
16 & 3 & 5, 5, 5 \\
\hline
17 & 4 & 3, 3, 5, 5 \\
\hline
18 & 4 & 2, 5, 5, 5 \\
\hline
18* & 4 & 3, 4, 5, 5 \\
\hline
19 & 4 & 3, 5, 5, 5 \\
\hline
20 & 4 & 4, 5, 5, 5 \\
\hline
21 & 4 & 5, 5, 5, 5 \\
\hline
22 & 4 & 5, 5, 5, 6 \\
\hline
23 & 4 & 5, 5, 5, 7 \\
\hline
24 & 5 & 3, 5, 5, 5, 5 \\
\hline
25 & 5 & 4, 5, 5, 5, 5 \\
\hline
26 & 5 & 5, 5, 5, 5, 5 \\
\hline
27 & 5 & 5, 5, 5, 5, 6 \\
\hline
28 & 5 & 5, 5, 5, 5, 7 \\
\hline
29 & 5 & 5, 5, 5, 6, 7 \\
\hline
30 & 4 & 3, 5, 5, 16 \\
\hline
31 & 4 & 4, 5, 5, 16 \\
\hline
32 & 4 & 5, 5, 5, 16 \\
\hline
33 & 4 & 5, 5, 6, 16 \\
\hline
34 & 4 & 5, 5, 7, 16 \\
\hline
35 & 4 & 5, 6, 7, 16 \\
\hline
36 & 5 & 5, 5, 5, 5, 15 \\
\hline
37 & 5 & 5, 5, 5, 5, 16 \\
\hline
38 & 5 & 5, 5, 5, 6, 16 \\
\hline
39 & 5 & 5, 5, 5, 7, 16 \\
\hline
40 & 5 & 5, 5, 6, 7, 16 \\
\hline
41 & 5 & 5, 5, 7, 7, 16 \\
\hline
42 & 4 & 7, 7, 11, 16 \\
\hline
43 & 5 & 5, 7, 7, 7, 16 \\
\hline
44 & 4 & 10, 11, 11, 11 \\
\hline
45 & 4 & 11, 11, 11, 11 \\
\hline
46 & 4 & 7, 11, 11, 16 \\
\hline
47 & 4 & 8, 11, 11, 16 \\
\hline
48 & 4 & 9, 11, 11, 16 \\
\hline
49 & 4 & 10, 11, 11, 16 \\
\hline
50 & 4 & 11, 11, 11, 16 \\
\hline
51 & 4 & 7, 11, 16, 16 \\
\hline
52 & 5 & 7, 11, 11, 11, 11 \\
\hline
53 & 4 & 9, 11, 16, 16 \\
\hline
54 & 4 & 10, 11, 16, 16 \\
\hline
55 & 4 & 11, 11, 16, 16 \\
\hline
56 & 5 & 11, 11, 11, 11, 11 \\
\hline
57 & 4 & 11, 13, 16, 16 \\
\hline
58 & 4 & 11, 14, 16, 16 \\
\hline
59 & 4 & 11, 15, 16, 16 \\
\hline
60 & 4 & 11, 16, 16, 16 \\
\hline
61 & 5 & 11, 11, 11, 11, 16 \\
\hline
62 & 4 & 13, 16, 16, 16 \\
\hline
63 & 4 & 14, 16, 16, 16 \\
\hline
64 & 4 & 15, 16, 16, 16 \\
\hline
65 & 4 & 16, 16, 16, 16 \\
\hline
66 & 4 & 16, 16, 16, 17 \\
\hline
67 & 4 & 16, 16, 16, 18 \\
\hline
68 & 5 & 11, 11, 13, 16, 16 \\
\hline
69 & 5 & 11, 11, 14, 16, 16 \\
\hline
70 & 5 & 11, 11, 15, 16, 16 \\
\hline
71 & 5 & 11, 11, 16, 16, 16 \\
\hline
72 & 5 & 11, 12, 16, 16, 16 \\
\hline
73 & 5 & 11, 13, 16, 16, 16 \\
\hline
74 & 5 & 11, 14, 16, 16, 16 \\
\hline
75 & 5 & 11, 15, 16, 16, 16 \\
\hline
76 & 5 & 11, 16, 16, 16, 16 \\
\hline
77 & 5 & 12, 16, 16, 16, 16 \\
\hline
78 & 5 & 13, 16, 16, 16, 16 \\
\hline
79 & 5 & 14, 16, 16, 16, 16 \\
\hline
80 & 5 & 15, 16, 16, 16, 16 \\
\hline
81 & 5 & 16, 16, 16, 16, 16 \\
\hline
\caption{Minimum weighted Szeged index trees.}
\label{tab:tree}
\end{longtable}

\begin{figure}[htb]
    \centering
    \includegraphics[width=0.4\textwidth]{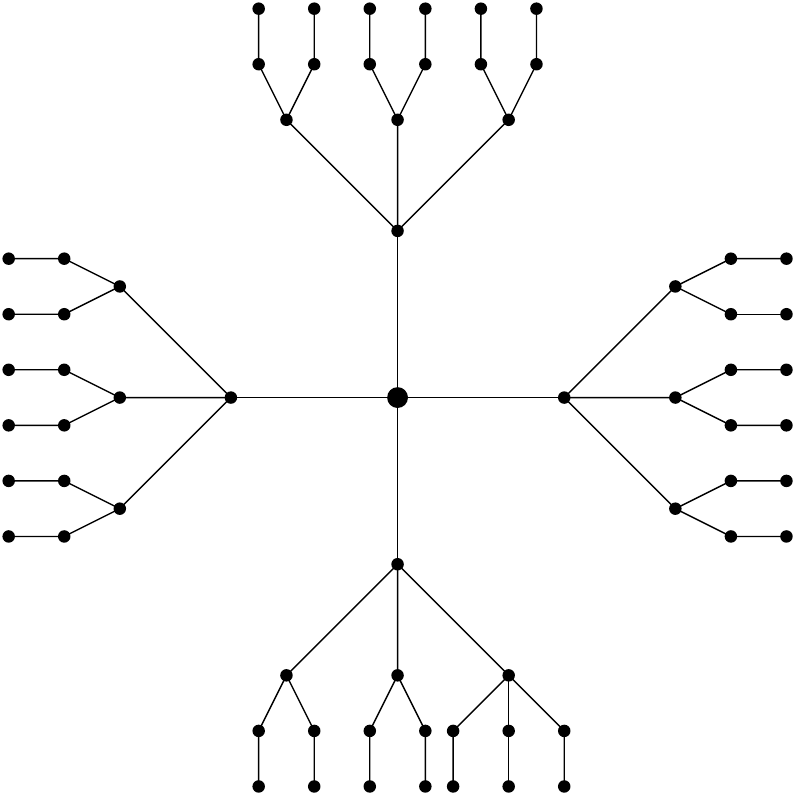}
    \caption{Best weighted Szeged index on 67 vertices.}
    \label{fig:5}
\end{figure}

\subsection{Regular Trees}
By a regular ending branch we mean a minimal (weighted Szeged index) ending branch whose children
are of the same order. In the first 80 minimal ending branches 1, 2, 3, 5, 7, 10, 11, 16, 22, 34, 49, 65
are the regular ones. Based on our observations we have the following conjecture. By a main branch we mean an ending branch that is directly connected to the root of the tree.
\begin{conjecture}
In a minimum weighted Szeged index tree all but at most one main ending branches are regular ending branches.
\end{conjecture}

The same definition applies to regular trees and the first regular minimal trees are
1, 2, 3, 5, 7, 9, 16, 21, 26, 45, 56, 65 and 81.

The simplified degree sequence of a regular ending branch is the degree sequence of vertices on a path from the root of the ending branch to a leaf. For example, the simplified degree sequence of a minimal ending branch of order 65 is 4, 3, 2, 1.

Our calculation shows that a minimal ending branch of order 326 (again our expectation) is not a regular ending branch with simplified degree sequence of 5, 4, 3, 2, 1. Actually an ending branch on 326 vertices with 3 children of order 103, 103, 119 works better than the regular ending branch.

We note that there seem to be no vertices of degree greater than 6 in the optimal trees.
In order to understand this phenomenon, we offer the following thought experiment.
We will calculate an approximation of the weighted Szeged index for complete $k$-ary
trees with a very high number $n$ of vertices.

Note that the root has degree $k$, while the other vertices have degree $k+1$. We will ignore this
complication in order to simplify the expressions, and treat the every edge as having degree sum
$2k+2$, so we can just calculate the ordinary Szeged index and multiply it by $2k+2$. The Szeged
index of a complete $k$-ary tree is estimated as follows.

Each edge from the root to one of its $k$ children contributes the product $\frac{n-1}{k} \left( n - \frac{n-1}{k} \right)$ because
the child has $\frac{n-1}{k}$ descendants. In this expression the $n^2$ term dominates, since we assumed that
$n$ is very large; its term is $\frac{k-1}{k^2} n^2$. There are $k$ edges like this, so their overall contribution to
the Szeged index is $\frac{k-1}{k} n^2 + o(n^2)$. Similarly, $k^2$ edges between the children and grandchildren
of the root contribute collectively $k^2 \left( \frac{n-k-1}{k^2} \right) \left( n - \frac{n-k-1}{k^2} \right)$ which equals
$\frac{k^2 - 1}{k^2} n^2 + o(n^2)$. A similar calculation for the remaining levels reveals that the Szeged
index is $n^2 \left( \frac{k-1}{k} + \frac{k^2 - 1}{k^2} + \frac{k^3 - 1}{k^3} + \dots \right) + o(n^2)$. Since there are $\log_k n$
levels, we can say roughly that the Szeged index of a complete $k$-ary tree is about $n^2 \log_k n$, and the
weighted Szeged index is about $n^2 (2k+2) \log_k n$. The expression can be analyzed, but since it only represents a rough approximation we can just look at a large value. When $n=10^6$, the expression
appears to be minimized around $k=4$; we feel this may explain why we don't see vertices of degree more than
6. In any event, a large $k$, say $k=10$, would be disadvantageous. We also see from the calculation that the
contributions of the lower levels of a complete $k$-ary tree are slightly  increasing as we go down the tree, and
perhaps this explains why the degrees in our optimal trees are decreasing. In large optimal
trees we have found that the degrees are decreasing just enough to make contributions of all levels roughly the
same (with the exception of the bottom three or four levels).

\section{Conclusions}
\label{sec:conc}

In view of Conjecture \ref{min-tree}, it is a good idea to
understand the structure of minimum weighted Szeged index
trees.   Even if it turns out to be false, knowing the
structure of such trees could give us some insight to
understand the structure of minimum weighted Szeged index
graphs in the general case.

In this work we introduced the concept of ending branch,
which we used to analyze the structure of minimum weight
Szeged index trees in a recursive fashion.   Our observations
were useful to computationally construct the trees on at most
81 vertices, extending the list of 25 trees given in
\cite{jan}.

Based on our results and experimental observations,
we finalize the section with some conjectures that,
if true, can give insights on the structure of minimum
weighted Szeged index trees. Also, they represent well
determined future lines of work that might be explored.
\begin{conjecture}
In a minimum weighted Szeged index tree, the degree sequence from the root to any leaf is non-increasing.
\end{conjecture}
\begin{conjecture}
Vertices of degree 1 are attached to vertices of degree at most 3.
\end{conjecture}
\begin{conjecture}
There are no vertices of degree greater than 6 in a minimum weighted Szeged index tree.
\end{conjecture}


\begin{thebibliography}{99}
\bibitem{bondy2008}
	J.A.~Bondy and U.S.R~Murty,
	Graph Theory,
	Springer, Berlin, 2008.

\bibitem{jan}
J. Boka, B. Furtulab, N. Jedli\v{c}kov\'a,
R. \v{S}krekovskid,
MATCH Commun. Math. Comput. Chem. 82 (2019) 93--109.

\bibitem{ref6} M. V. Diudea, I. Gutman, L. J\"antschi,
Molecular Topology, Nova, Huntington, 2002.

\bibitem{ref2}
E. Estrada,
Characterization of 3D molecular structure,
Chem. Phys. Lett. 319 (2000) 713--718.

\bibitem{ref3} I. Gutman, B. Zhou,
Laplacian energy of a graph,
Linear Algebra Appl. 414 (2006) 29--37.

\bibitem{ABCJCTB}
S. A. Hosseini, B. Mohar, M. B. Ahmadi, The evolution of the structure of ABC-minimal trees, arXiv e-prints, page arXiv:1804.02098, 2018.

\bibitem{krag_tree}
S. A. Hosseini, M. B. Ahmadi, I. Gutman,
Kragujevac trees with minimal atom-bond
connectivity index,
MATCH Commun. Math. Comput. Chem. 71 (2014) 5--20.

\bibitem{ilic}
A. Ili\'{c}, N. Milosavljevi\'{c}, The weighted vertex PI index, Math. Comput. Modell. 57
(2013) 623--631.

\bibitem{ref4} G. Indulal, I. Gutman, A. Vijaykumar,
On distance energy of graphs,
MATCH Commun. Math. Comput. Chem. 60 (2008) 461--472.

\bibitem{ref5} J. Liu, B. Liu,
A Laplacian energy like invariant of a graph,
MATCH Commun. Math. Comput. Chem. 59 (2008) 355--372.

\bibitem{ref10}
S. Nagarajan, K. Pattabiraman, M. Chandrasekharan, Weighted Szeged index of
generalized hierarchical product of graphs, Gen. Math. Notes 23 (2014) 85--95.

\bibitem{ref11} K. Pattabiraman, P. Kandan, Weighted Szeged indices of some graph operations,
Trans. Comb. 5 (2016) 25--35.

\bibitem{ref12} K. Pattabiraman, P. Kandan, Weighted Szeged index of graphs, Bull. Int. Math.
Virtual Inst. 8 (2018) 11--19.

\bibitem{ref1}
R. Todeschini, V. Consonni,
Molecular descriptors for chemoinformatics, Wiley-VCH, Weinheim, 2009.

\end{thebibliography}
\end{document}